\documentclass[10pt,reqno]{amsart}
    \usepackage{geometry}
\newtheorem{definition}{Definition}[section]

\newtheorem{remark}[definition]{Remark}
\newtheorem{example}[definition]{Example}
\newtheorem{lemma}[definition]{Lemma}
\newtheorem{proposition}[definition]{Proposition}
\newtheorem{theorem}[definition]{Theorem}
\newtheorem{corollary}[definition]{Corollary}
\def\P{{\mathbb{P}}}

\def\N{{\mathbb{N}}}
\def\K{{\mathbb{K}}}

\def\Syzy{{\mbox{Syz}}}

\begin{document}
\title{Minimal solutions of the rational interpolation problem}
\author{Teresa Cortadellas Ben\'itez}
\address{Universitat de Barcelona, Facultat d'Educaci\'o.
Passeig de la Vall d'Hebron 171,
08035 Barcelona, Spain}
\email{terecortadellas@ub.edu}

\author{Carlos D'Andrea}
\address{Universitat de Barcelona, Departament de Matem\`atiques i Inform\`atica,
  Gran Via de les Corts Catalanes 585,
 08007 Barcelona,
 Spain} \email{cdandrea@ub.edu}
\urladdr{http://www.ub.edu/arcades/cdandrea.html}

\author{Eul\`alia Montoro}
\address{Universitat de Barcelona, Departament de Matem\`atiques i Inform\`atica,
 Gran Via de les Corts Catalanes 585,
 08007 Barcelona,
 Spain}
\email{eula.montoro@ub.edu}

\thanks{T.~Cortadellas is supported by the Spanish research project  MTM2016-78881-P, C.~D'Andrea and E.~Montoro are supported by the Spanish MINECO/FEDER
research project MTM 2015-65361-P. C.~D'Andrea is also supported by the ``Mar\'ia de Maeztu'' Programme for Units of Excellence in R\&D (MDM-2014-0445), and by the European Union's Horizon 2020 research and innovation programme under the Marie Sklodowska-Curie grant agreement No. 675789.}

%\date{\today}



\begin{abstract}
We explore connections between the approach of solving the rational interpolation problem via resolutions of ideals and syzygies with the standard method provided by the Extended Euclidean Algorithm.  As a consequence, we obtain explicit descriptions for solutions of  ``minimal'' degrees in terms of the degrees of elements appearing in the EEA.  This allows us to describe the minimal degree in a $\mu$-basis of a polynomial planar parametrization in terms of  a ``critical'' degree arising in the EEA.
\end{abstract}
\maketitle

\section{Introduction}
Let $\K$ be a field, $l,n_1,\dots, n_l$ positive integers, $n:=n_1+\cdots +n_l$, and
\begin{equation}
\label{data}
 (x_i, y_{i,j})\in\K^2, \mbox{ for } i=1,\dots l,\quad j=0,\dots n_i-1
 \end{equation}
 with  $x_i\neq x_j$ if $i\neq j.$ An interpolating rational function associated to this data is a function  $y(x)\in\K(x)$ satisfying

\begin{equation}
\label{interpolant}
 y^{(j)}(x_i)=y_{i,j}, \qquad i=1,\dots l, \quad j=0,\dots n_i-1.
 \end{equation}
In particular, we are requiring that the rational function  $y(x)$ is defined on all the points $x_i,\, i=1,\ldots, l.$

The {\em rational interpolation problem} asks to describe  all the rational functions verifying \eqref{interpolant}.   Note that in principle the well known interpolating polynomial associated to the data \eqref{data} is always a solution of \eqref{interpolant}, so this problem is always solvable. Are there more solutions? How many? Can we ``parameterize'' them? Is there a ``minimal'' or ``compact'' solution?

The rational interpolation has been well studied in the last centuries, with references going back to the mid 1800's (\cite{ca41,ros45,pre53}).  In the last decades, the interest in this problem focused in more algorithmic and computational aspects due to the increasing growth of these areas of research and applications:  see for instance \cite{Sal62, Kah69, Wuy75, BL00, TF00} and the references therein. A unified framework, which relates the rational interpolation problem with the Euclidean Algorithm, is presented in \cite{ant88}, and also in the book \cite[Section 5.7]{vzgg13}, where it is called called \emph{rational function reconstruction.} There are also  explicit closed formulae in terms of the input data that can be derived by operating with symmetric operators or subresultants, see  \cite{las03,DKS15}.

%Classically, most of these problems can be solved with the aid of basic and elementary Linear Algebra. However, from a computational point of view, dealing with matrices and %computing their kernels (which one must do several times in order to get the solutions) has high complexity. The use of the EEA makes computations faster and also descriptions of the %possible obstructions (like ``jumps'' in the degrees of interpolation) neater. Our results can be regarded as a contribution in that direction, as we show that minimal degrees (and %degrees of $\mu$-bases) can be computed ``fast''.

In this paper, we focus on the possible {\em degrees} that a solution of \eqref{interpolant} can have, and detect the minimal of these degrees in optimal time. To be more precise, a  {\em degree} for a rational function $y(x)=\frac{a(x)}{b(x)}$, with coprime $a(x), b(x)\in\K[x]$, can be  defined, among others, as
 \begin{align*}
\delta (y(x)) =& \max \{ \deg a(x), \deg b(x)\} \\ \intertext{ or }
\kappa (y(x)) =& \deg a(x) +\deg b(x).
\end{align*}

Fixed one of these degrees, we  say that $\gamma\in\N$ is an {\em admissible degree} if there exists an interpolating function of degree $\gamma$. A natural question to ask is which numbers are admissible degrees, and in particular to characterize the smallest of them. A minimal solution of the rational interpolation problem is a solution of minimal degree.  Also, it is of interest to parameterize all the admissible rational functions.
For instance, for the data

\[
\begin{array}{lll}
x_1=0, &y_{1,0}=-2 & \\
x_2=2, &y_{2,0}=6& \\
x_3=-1, &y_{3,0}=-3 \\
x_3=-1, & y_{3,1}=3 ,
\end{array}
\]
it turns out that its  $\delta$-minimal degree is $2$ and the fractions $\frac{-2-3\lambda x^2}{\frac{-x^2}{3}+1-\lambda x}$ define the  minimal interpolating functions for any $\lambda\in \K$, $\lambda\neq -\frac{1}{6}, -\frac{2}{3}$. Moreover, the interpolating functions $y(x)=\frac{a(x)}{b(x)}$ with $\delta (y(x))=\delta>2$ can be parameterized as
\begin{equation*}
y(x)=\frac{(\lambda_0+\cdots +\lambda_{\delta-2}x^{\delta-2})r_2(x)+ (\lambda'_0+\cdots +\lambda'_{\delta-3}x^{\delta-3} + x^{\delta-2} )r_3(x)  }{(\lambda_0+\cdots +\lambda_{\delta-2}x^{\delta-2})s_2(x)+ (\lambda'_0+\cdots +\lambda'_{\delta-3}x^{\delta-3} + x^{\delta-2})s_3(x)}
\end{equation*}
for suitable polynomials $r_2(x), r_3(x), s_2(x), s_3(x)\in\K[x]$
provided that the denominator does not vanish. This example appears in \cite[Example $3.8a$]{ant88}, see also Example \ref{example1} for details on these calculations.

The problem of describing the admissible degrees for the interpolation problem was tackled in \cite{antand86} for $\delta$ and, in \cite{ant88} for $\kappa.$ The main tool in the first mentioned paper is a divided-differences matrix, whereas the latter uses  the Euclidean Algorithm to solve this problem. Later, in \cite{ABKW90} the solutions of the classical (i.e. when $n_i=1 \ \forall i=1,\ldots, l$) rational interpolation problem were given as the kernel of a matrix encoding the data of the problem. In \cite{ra97} this matrix was homogenized, and  the numerical invariants of a minimal free resolution of its cokernel  used to tackle the problem.

In  this paper, we connect the results obtained in \cite{ra97} via resolutions of ideals and syzygies with the now standard and fast approach given by the Extended Euclidean Algorithm.
As a consequence, we give explicit descriptions for both $\delta$ and $\kappa$ minimal degrees in terms of the degrees of elements appearing in the EEA. Of course if one is interested in a ``fast method'' to solve any instance of the interpolation problem, the EEA is the most efficient tool available, and any improvement in dealing with this problem will also get translated into a faster algorithm to solve the EEA. In this sense, the two situations (solving the rational interpolation problem and computing the EEA of two polynomials) are equivalent from an algorithmic and complexity  point of view. But this is not our focus. We are interested in exploring connections with syzygies and free resolutions because in the last decades the latter tools have been used to deal with other kind of geometric problems like implicitization of rational parametrizations, or the description of some invariantes associated to them. Hence, any dictionary between these methods brings the potential of shedding some light to situations different than the rational interpolation problem per se. As an example of this, in Section \ref{mb}  we obtain a simple description  of the value of ``$\mu$'' for a $\mu$-basis of a polynomial plane parameterization. This also shows that one can compute both $\mu$ and the $\mu$-basis in considerable fast time. It would be interesting to explore whether this result can be generalized for any rational parametrization, as it would reduce the complexity of computing $\mu$ and $\mu$-bases considerably. For the latest state of the art in this area, see \cite{HHK17} and the references therein.

We also treat with our methods the {\em Hermite rational interpolation problem}, which consists in, for a given $d\in\N,\, 0\leq d<n,$ decide if there exist, and if so compute, polynomials $a(x),b(x)$ of degrees bounded by $d$ and $n-d-1$ respectively, such that $\frac{a(x)}{b(x)}$ interpolates the data (see \cite{CDM18} for more on this problem). This is done at the end of Section \ref{ccappa}.

The paper is organized as follows: in Section \ref{EA} we review some basics on the Hermite interpolating polynomial and the Euclidean Algorithm.  In Section \ref{3}, we introduce the language of syzygies which help us obtain a ``minimal basis'' associated to the rational interpolation problem. We show in Theorem \ref{nininal} -which is a generalization of \cite[Theorem 2.11]{ABKW90} to the case of interpolation with multiplicities- that these minimal bases allow us to make the $\delta$-minimal degree explicit. Then we turn to make even more explicit this invariant by means of the EEA. This is the content of Theorem \ref{mtt}, where we extract a minimal basis from some ``critical value'' in the sequence of degrees in the EEA.

In Section \ref{mb} we apply these results to make explicit the value of $\mu$ for a $\mu$-basis of a plane polynomial parametrization. In Theorem \ref{mu}
such $\mu$ is expressed as a critical value in the sequence of degrees in a suitable EEA. We conclude the paper by applying our tools to study the $\kappa$-degree in Section \ref{ccappa}, where we recover  the results of Antoulas in \cite{ant88}  with our methods.

%\bigskip

%{\bf Acknowledgements}
%T.~Cortadellas is supported by the Spanish research project MTM2016-78881-P, C.~D'Andrea and E.~Montoro are supported by the Spanish MINECO/FEDER
%research project MTM 2015-65361-P. C.~D'Andrea is also supported by the ``Mar\'ia de Maeztu'' Programme for Units of Excellence in R\&D (MDM-2014-0445), %and by the European Union's Horizon 2020 research and innovation programme under the Marie Sklodowska-Curie grant agreement No. 675789.

\bigskip
\section{Hermite interpolating polynomial and Euclidean algorithm}\label{EA}

All along this paper $\K$ is an infinite field of characteristic either zero or larger than  $\max \{n_1,\dots n_l\}$. This assumption certifies the existence of the {\em Hermite interpolating polynomial} associated to \eqref{data}, which is the unique interpolating polynomial of  this data having degree smaller than $n.$  We denote it by $g(x)$.

It is straightforward to check that any rational interpolating function $y(x)=\frac{a(x)}{b(x)}$ satisfies
\begin{equation}
\label{weak}
(a(x)-b(x)g(x))^{(j)}(x_i)=0, \quad i=1,\dots l, \quad j=0,\dots n_i-1,
\end{equation}
and so,
$(x-x_i)^{n_i}$ divides $a(x)-b(x)g(x)$ for $ i=1,\dots l$. Thus, if we set $f(x):=(x-x_1)^{n_1}\cdots (x-x_l)^{n_l}$, there exists $c(x)\in\K[x]$ such that
\begin{equation}
\label{eq}
 a(x)=b(x)g(x) +c(x)f(x).
 \end{equation}
 In fact, \eqref{weak} and  \eqref{eq} are equivalent,
 and we refer to them, indistinctly, as {\em the weak interpolation conditions.} We have that, $y(x)=\frac{a(x)}{b(x)}$ is a interpolating function if and only if  the pair $(a(x),b(x))$ satisfies the weak interpolation conditions and $b(x_i)\neq 0$, for $i=1,\dots, l$.

The following result follows straightforwardly.
\begin{proposition}\label{above}
If the pair $(a(x),b(x))$ defines an interpolating function then, for any polynomial $c(x)$ without roots in the set $\{x_1,\dots, x_l\}$, the pair $(a(x)c(x),b(x)c(x))$ also defines an interpolating function. Reciprocally, if $(a(x)c(x),b(x)c(x))$ defines an interpolating function, then $(a(x),b(x))$ defines an interpolating function.
\end{proposition}
Another way of stating Proposition \ref{above} is that there exists an interpolating function of the data in a given class of $\K(x),$ if and only if the irreducible fraction of the class is an interpolating fraction. In particular, if $y(x)$ is an interpolating function, then there exist coprime  polynomials $a(x)$, $b(x)$ such that $y(x)=\frac{a(x)}{b(x)}$ in $\K(x)$.

Consider now two polynomials $r_0(x),\,r_1(x)\in\K[x]$ with $\deg r_0(x)\geq \deg r_1(x)\geq0$. Following the {\em Euclidean Algorithm}, there exists a positive integer $N$ and unique nonzero polynomials $r_2(x),\ldots, r_N(x),\, q_1(x),\ldots, q_N(x)$ with $\deg r_i(x) > \deg r_{i+1}(x)$, $i=1,\dots N$,
such that
\begin{equation}\label{euc} r_i(x)=q_{i+1}(x)r_{i+1}(x) +r_{i+2}(x), \qquad i=0,\dots ,N-1, \end{equation}
where $N$ is such that $r_N(x)\neq0$ and $r_{N+1}(x)=0$.

 Note that $\deg q_i(x)>0$ for all $i>1$ and that $\deg q_i(x)>0$ for all $i\geq 1$ if $\deg r_0(x)> \deg r_1(x)$.

Using the quotients $q_i(x),$  we define recursively two sequences of polynomials $s_i(x)$, $t_i(x)$, for $i=0,\dots N+1$,
 %with $\deg s_i(x)> \deg s_{i-1}(x)$, $\deg t_i(x)> \deg t_{i-1}(x)$
  as follows:
\begin{equation}\label{abb}
\begin{split}
& \begin{pmatrix} s_{0}(x)\\t_{0}(x) \end{pmatrix}= \begin{pmatrix} 0\\1 \end{pmatrix}, \quad \begin{pmatrix} s_{1}(x)\\t_{1}(x) \end{pmatrix}= \begin{pmatrix} 1\\0 \end{pmatrix}, \\  &\begin{pmatrix} s_{i}(x)\\t_{i}(x) \end{pmatrix}= \begin{pmatrix} s_{i-2}(x)\\t_{i-2}(x) \end{pmatrix}- q_{i-1}(x) \begin{pmatrix} s_{i-1}(x)\\t_{i-1}(x) \end{pmatrix},
 \end{split}
 \end{equation}
 with $\deg s_i(x)> \deg s_{i-1}(x)$ and $\deg t_i(x)> \deg t_{i-1}(x)$ for $i>2$.
We deduce from \eqref{euc} and \eqref{abb} that
\begin{equation*}
  r_i(x)= r_1(x)s_i(x) +r_0(x)t_i(x), \quad i=0,\dots N+1.
\end{equation*}
%Note that $\deg q_i(x)>0$ for all $i>1$.
\begin{remark}
\label{degrees}
An easy inductive argument on \eqref{euc} and \eqref{abb} yields
\begin{enumerate}
\item [(a)] $\deg r_i(x)= \deg r_0(x)-(\deg q_1(x) + \cdots + \deg q_i(x))$, for $i=1,\dots, N$.
\item [(b)] $\deg s_i(x)=\deg q_1(x)+ \cdots + \deg q_{i-1}(x)$, for $i=2,\dots, N$.
\end{enumerate}
\end{remark}

\begin{lemma}
\label{emes}
\cite[Lemma 2.5]{ant88} Assume that $\deg r_0(x) > \deg r_1(x)$. For every $a(x),b(x),c(x)\in\K[x]$ satisfying  $ a(x)= r_1(x)b(x) +r_0(x)c(x)$, there exist unique polynomials $m_0(x)$, $m_1(x),\dots , m_{N+1}(x)$ such that

\[ \begin{pmatrix} a(x)\\ b(x)\\c(x) \end{pmatrix}= \sum_{i=0}^{N+1}m_i(x) \begin{pmatrix} r_i(x)\\s_{i}(x)\\t_{i}(x) \end{pmatrix}
\]
with $\deg m_i(x)<\deg q_i(x)$, $i=1,\dots ,N$.
\end{lemma}
Note that there are no bounds on the degrees of  neither $m_0(x)$ nor $m_{N+1}(x)$ in Lemma \ref{emes}.

\bigskip
\section{Syzygies and the $\delta$-degree}\label{3}
For an interpolating function  $y(x)=\frac{a(x)}{b(x)}$ of  \eqref{data} there exists, by using \eqref{eq}, $c(x)\in\K[x]$ such that
$a(x)=g(x)b(x)+c(x)f(x)$. So, $(a(x),b(x),c(x))\in \K[x]^3$ is in the kernel of the following morphism of $\K[x]$-modules:
\begin{equation}\label{avv}
\varphi:\K[x]^3 \xrightarrow{{\tiny (1\, -g(x) \, -f(x))}} \K[x].
\end{equation}
The following result is a straightforward consequence of the application of the Euclidean Algorithm to the solution of an equation of the form
\begin{equation*}
a(x)=g(x)b(x)+c(x)f(x).
\end{equation*}
\begin{proposition}\label{free}
The kernel of $\varphi$ is a free module of rank
$2,$ and it has \linebreak $\{(f(x),0,1),\,(g(x),1,0)\}$ as a basis.
\end{proposition}
Note that, in particular we have
\begin{equation*}
(a(x),b(x))= \frac{a(x)-b(x)g(x)}{f(x)}(f(x),0) +b(x)(g(x),1),
\end{equation*}
and that  from here we deduce  straightforwardly that $(f(x),0)$ and $(g(x),1)$ is also a basis of  the $\K[x]$-submodule of $\K[x]^2$ defined by
\begin{equation*}
 Y:= \{(a(x), b(x))\in \mathbb{K}[x]^{2}\mbox{ such that } a(x)-b(x)g(x) \in  f(x)\K[x] \}.
\end{equation*}
From its definition, it is clear that that  $Y$ is the set of the pairs of polynomials satisfying the weak interpolating conditions \eqref{weak} or \eqref{eq}.

Homogenizing the above situation with a second variable $z,$ we obtain a homogeneous morphism
$\phi: \K[x,z](-n)^3 \longrightarrow \K[x,z]$ given by the matrix
\[\begin{pmatrix} z^n&-g(x,z) &-f(x,z) \end{pmatrix}, \]
and, by the {\em Hilbert Syzygy Theorem}, a minimal free resolution of $\K[x,z]/\mbox{Coker}(\phi)$	
\begin{equation}\label{HST}
0 \longrightarrow \K[x,z](-n-\mu_1)\oplus \K[x,z](-n-\mu_2)\longrightarrow \K[x,z](-n)^3 \stackrel{\phi}{\longrightarrow }\K[x,z]
\end{equation}	
with $\mu_1+\mu_2=n,$ and  $\mu_1\leq \mu_2$. Our goal is to make explicit $\mu_1$ and $\mu_2$ by means of syzygies and the Euclidean Algorithm.

Let   $\{(a_1(x,z),b_1(x,z),c_1(x,z)),\,(a_2(x,z),b_2(x,z),c_2(x,z))\}$ be a basis of the kernel of $\phi$ such that
$\deg(a_i(x,z),b_i(x,z),c_i(x,z))=\mu_i,\,i=1,2.$ Any relation or syzygy among $z^n$, $-g(x,z)$ and $-f(x,z)$ can be written uniquely  as a polynomial combination of these two triplets. In particular, for any homogeneous syzygy $(a(x,z),b(x,z),c(x,z))$ of degree $\delta$ there exist unique homogeneous polynomials $p(x,z)$, $q(x,z)$, of degrees $\delta-\mu_1$ and $\delta-\mu_2$ respectively, or zero polynomials, such that {\small
\begin{multline}
\label{homog}
(a(x,z),b(x,z),c(x,z)) = p(x,z)\cdot (a_1(x,z),b_1(x,z),c_1(x,z))\\ +q(x,z)\cdot (a_2(x,z),b_2(x,z),c_2(x,z)).
\end{multline}

Set now  $a_i(x)=a_i(x,1)$, $b_i(x)=b_i(x,1)$, $c_i(x)=c_i(x,1),\, i=1, 2.$
\begin{lemma}\label{lema}
 For any $(a(x),b(x),c(x))$ in the kernel of $\varphi$, we have that $$\delta:=\max \{\deg a(x), \deg b(x), \deg c(x)\}=\max \{\deg a(x), \deg b(x)\},$$ and there
 exist unique polynomials $p(x)$, $q(x)$ such that
\begin{equation}\label{dehom}
(a(x),b(x)) = p(x)(a_1(x),b_1(x))+q(x)(a_2(x),b_2(x)).
\end{equation} Moreover, $\deg p(x)\leq \delta -\mu_1$ and $ \deg q(x) \leq \delta-\mu_2$.
\end{lemma}
\begin{proof}
Let $k=\deg g(x).$ Note that $k\leq n-1<n.$

Suppose that $\deg c(x)>\max\{\deg a(x),\,\deg b(x)\}.$ Then we would have that
$$\deg\left(c(x)f(x)\right)=\deg c(x)+n>\max\{\deg a(x), \deg b(x) +k\}\geq \deg\left(a(x)-b(x)g(x)\right),
$$
a contradiction with \eqref{eq} so, the first part of the claim holds. For the second, we homogeneize $a(x),\,b(x),\,c(x)$ to degree $\delta$ and get \eqref{homog} with  $p(x,z),\,q(x,z)\in\K[x,z]$ of respective degrees $\delta-\mu_1$ and $\delta -\mu_2.$ To get \eqref{dehom} we set $z=1,$ and the claim follows straightforwardly.
\end{proof}

\begin{definition}
We say that a basis $\{(a_1(x),b_1(x)),\, (a_2(x),b_2(x))\}$ of $Y$ is {\em minimal}  (or of minimal degree) if $\max\{\deg a_i(x),\deg b_i(x)\}=\mu_i,\ i=1,2.$
\end{definition}
Note that  in general, $\{(f(x),0),\,(g(x),1)\}$ is not a minimal basis.

\begin{remark}
\label{unique}
Let $\{(a_1(x),b_1(x))$, $(a_2(x),b_2(x))\}$ a minimal basis of $Y$.  From its definition, we deduce that for any $(a(x),b(x))\in Y,$ there exist unique $p(x)$, $q(x)\in\K[x]$ such that
\begin{equation*}
\begin{split}
(a(x),b(x)) & = p(x)(a_1(x),b_1(x))+q(x)(a_2(x),b_2(x)) \\ & =  (p(x)a_1(x)+q(x)a_2(x), p(x)b_1(x)+q(x)b_2(x)).
\end{split}
\end{equation*}
Let $c(x)\in\K[x]$ be the polynomial such that \eqref{eq} holds, and set $\delta:=\max\{\deg a(x),\,\deg b(x)\}.$ Homogeneizing this situation, thanks to Lemma \ref{lema} we have
\begin{equation*} \delta  =\max\{ \deg(p(x,z)a_1(x,z)+q(x,z)a_2(x,z)), \deg(p(x,z)b_1(x,z)+q(x,z)b_2(x,z))\}, \end{equation*}
and from here we deduce that  $ \deg p(x)\leq \delta-\mu_1$ and $ \deg q(x)\leq \delta-\mu_2$. In particular,
 $\delta(\frac{a(x)}{b(x)}) \geq \mu_1,$ and if $q(x)\neq 0$, then  $\delta(\frac{a(x)}{b(x)}) \geq \mu_2$.

\end{remark}

\begin{remark}\label{any}
Any basis of $Y$ with $\delta$-degrees $\delta_1$, $\delta_2$ with $\delta_1 +\delta_2= n$ is a minimal basis of $Y$. In this case $\{\delta_1,\delta_2\}=\{\mu_1,\mu_2\}$.
\end{remark}

\begin{remark}
\label{factorsbasis}
Let $\{(a_1(x),b_1(x)),\,(a_2(x),b_2(x))\}$ be a minimal basis of $Y$. From Proposition \ref{above} we deduce that  the unique possible irreducible common factors of $a_i(x)$ and $b_i(x),$  are of the form $x-x_j$ for some $j=1,\ldots, l;\, i=1,2.$ Otherwise,  by removing such a factor, we would obtain a pair verifying the weak interpolating conditions with degree strictly smaller than $\mu_1$ or than $\mu_2$.

Also, since there exist interpolating functions for any data (for instance, the interpolating polynomial $g(x)$),   $b_1(x)$ and $b_2(x)$ cannot both vanish at  $x_i$ for any $i=1,\ldots, l.$
\end{remark}

%\begin{lemma}
%\label{unique} If $y(x)=\frac{a(x)}{b(x)}$ and $z(x)=\frac{c(x)}{d(x)}$ satisfy the weak interpolating conditions and $%\delta(y(x))+\delta(z(x))<n$ then $y(x)=z(x) $ in $\K(x)$.
%\end{lemma}
%\begin{proof}
%The polynomial $p(x)=a(x)d(x)-b(x)c(x)$ has degree bounded by $n-1$ and its derivatives $p(x)^{(j)}$ vanishes in the different %points $x_i$, for $i=1,\dots, l$ and $j=0,\dots, n_{i-1}$. Hence $p(x)=0$ and the rational fractions are equivalent.
%\end{proof}
The following is the main result of this section, which generalizes  \cite[Theorem 2.11]{ABKW90} to the case of interpolation with multiplicities.

\begin{theorem}\label{nininal}
Let  $\{(a_1(x),b_1(x)),\,(a_2(x),b_2(x))\}$ be a minimal basis of $Y$. The rational function $y(x)$ is an interpolating function if and only if there exist polynomials $p(x)$ and $q(x)$ such that
\[y(x)=\frac{p(x)a_1(x)+q(x)a_2(x)}{p(x)b_1(x)+q(x)b_2(x)} \]
with $p(x_i)b_1(x_i)+q(x_i)b_2(x_i)\neq 0$ for $i=1,\dots ,l$.

If $a_1(x)$ and $b_1(x)$ are coprime,  and $\mu_1<\mu_2$, then there is a unique interpolating function $y_{\min }(x)$ of minimal degree $\mu_1$ given by
\[  y_{\min }(x)=\frac{a_1(x)}{b_1(x)}.\]
Otherwise,
there is a family of interpolating functions of minimal degree $\mu_2$ which can be parametrized as
\[  y_{\min }(x)=\frac{a_2(x)+p(x)a_1(x)}{b_2(x)+p(x)b_1(x)},\]
where $\deg p(x)=\mu_2-\mu_1$, and $b_2(x_i)+p(x_i)b_1(x_i) \neq 0$, $i=1,\dots ,l$.
\end{theorem}

\begin{proof}
The rational function $y(x)=\frac{a(x)}{b(x)}$ interpolates \eqref{data} if and only if $b(x_i)\neq0$ for all $i=1,\ldots, l,$ and it satisfies the weak interpolation conditions of \eqref{eq}. Equivalently we must have $(a(x),b(x))\in Y$,  and $b(x_i)\neq 0$ for $i=1, \dots, l$. From here we deduce the first part of the claim thanks to Remark \ref{unique}.

Assume now that $\mu_1< \mu_2$ and $a_1(x)$ and $b_1(x)$ are coprime. If $b_1(x_i)=0$ for some $i$, then also $a_1(x_i)=0$, since $a_1(x)-b_1(x)g(x) \in f(x)\K[x]$ which contradicts the assumption. The rational function $y(x)=\frac{a_1(x)}{b_1(x)}$ interpolates the data in this case. Moreover, by Remark \ref{unique}, it is the unique rational function with minimum degree that interpolates the data.

If $\mu_1< \mu_2$ and $a_1(x)$ and $b_1(x)$ are not coprime then, by using Remark \ref{factorsbasis}, $b_1(x_i)=0$ for some $i=1,\dots ,l$ and $\frac{a_1(x)}{b_1(x)}$ doesn't interpolate the data. In this case, any pair
$(p(x)a_1(x)+q(x)a_2(x), p(x)b_1(x)+q(x)b_2(x))$ defining an interpolating function must have $q(x)\neq 0$ and degree at least $\mu_2$. So, the interpolating functions of minimal degree are parameterized by $\frac{a_2(x)+p(x)a_1(x)}{b_2(x)+p(x)b_1(x)}$, where $p(x)$ is a  polynomial of degree $\mu_2-\mu_1$ provided that $b_2(x_i)+p(x_i)b_1(x_i) \neq 0$, $i=1,\dots ,l$.
It is straightforward to verify that there is at least one of such polynomials. Indeed, by choosing $\lambda\in\K\setminus\{-\frac{b_2(x_i)}{b_1(x_i)},\,i=1,\ldots, l,\,b_1(x_i)\neq0\}$  which can be done thanks to Remark \ref{factorsbasis}, and the fact that $\K$ has infinite elements, the polynomial $b_2(x)+\lambda b_1(x)$ satisfies the claim. Note that for these values of $\lambda$ it cannot happen that $\delta\left(\frac{a_2(x)+\lambda a_1(x)}{b_2(x)+\lambda b_1(x)}\right)<\mu_2,$ as this would imply that the pair
$(a_2(x)+\lambda a_1(x), b_2(x)+\lambda b_1(x))$ is a polynomial multiple of $(a_1(x), b_1(x))$ which is impossible as $b_1(x)$ vanishes in some of the $x_i's$ and $\lambda$ has been chosen in such a way that $b_2(x)+\lambda b_1(x)$ does not vanish in any of these points.

If $\mu_1=\mu_2$  the family of interpolating rational functions of minimal degree can be written as
$\frac{a_2(x)+\lambda\,a_1(x)}{b_2(x)+\lambda\,b_1(x)}$, with $\lambda\in\K\setminus\{-\frac{b_2(x_i)}{b_1(x_i)},\,i=1,\ldots, l,\,b_1(x_i)\neq0\},$ as before. This concludes with the proof of the Theorem.
\end{proof}
\begin{corollary}
The minimal $\delta$-admissible degree is either $\mu_1$ or $\mu_2.$ The set of $\delta$-admissible degrees are either $\{\mu_1\}\cup\{\delta \geq\mu_2\},$ or $\{\delta \geq\mu_2\}.$
\end{corollary}
\begin{proof}
By Theorem \ref{nininal}, the minimal admissible $\delta$-degree of $y(x)$ is either $\mu_1$ or $\mu_2$. Moreover, any $y(x)=\frac{a(x)}{b(x)}$ interpolating function with $\delta$-degree $\geq\mu_2$ can be written as
\begin{equation}\label{frac}
 \frac{a(x)}{b(x)}= \frac{p(x)a_1(x)+ q(x)a_2(x)}{p(x)b_1(x)+ q(x)b_2(x)},
 \end{equation}
for some $p(x),\,q(x)\in\K[x],\, q(x)\neq 0$.

We have already seen in Theorem \ref{nininal} that $\mu_1$ may be admissible, and also that $\mu_2$ is admissible if $\mu_1$ is not.
Moreover, as in the proof of this Theorem, it is straightforward to check that there cannot be any rational function as in \eqref{frac} of $\delta$-degree strictly larger than $\mu_1$ and smaller than $\mu_2$, as this would imply
$(p(x)a_1(x)+ q(x)a_2(x), p(x)b_1(x)+ q(x)b_2(x))$ being a polynomial multiple of $(a_1(x), b_1(x))$, and hence it would have $\delta$-degree equal to $\mu_1,$ a contradiction.

To see that any $\delta\geq\mu_2$ is admissible, pick
$\frac{\tilde{a}(x)}{\tilde{b}(x)}$ a minimal solution of $\delta$-degree either $\mu_1$ or $\mu_2$. Note that this implies $\gcd(\tilde{a}(x),\tilde{b}(x))=1.$ Set now
$$(a(x),b(x)):=\left(\tilde{a}(x)+\lambda\,x^{\delta-\mu_2}a_2(x), \tilde{b}(x)+\lambda\,x^{\delta-\mu_2}b_2(x)\right),$$
for a suitable $\lambda\in\K\setminus\{0\}$ such that
\begin{equation}\label{gcd}
\gcd\left(\tilde{a}(x)+\lambda\,x^{\delta-\mu_2}a_2(x), \tilde{b}(x)+\lambda\,x^{\delta-\mu_2}b_2(x)\right)=1.
\end{equation}
This can be done because \eqref{gcd} is equivalent to the fact that the resultant of the polynomials $\tilde{a}(x)+\lambda\,x^{\delta-\mu_2}a_2(x)$ and  $\tilde{b}(x)+\lambda\,x^{\delta-\mu_2}b_2(x)$ does not vanish identically. This resultant is a polynomial in $\lambda$ whose constant coefficient is equal to $\mbox{Res}(\tilde{a}(x),\tilde{b}(x))\neq0$ as these polynomials do not share any common factor. So, $$\mbox{Res}\left(\tilde{a}(x)+\lambda\,x^{\delta-\mu_2}a_2(x), \tilde{b}(x)+\lambda\,x^{\delta-\mu_2}b_2(x)\right)\in\K[\lambda]$$ is not the zero polynomial, and by choosing $\lambda\in\K$ which is not a zero of this polynomials (this can be done because $\K$ is infinite), the claim follows for $\delta\geq\mu_2.$
\end{proof}

In what follows, we are going to make explicit the $\mu_i$'s by means of the Extended Euclidean Algorithm.

\subsection{Minimal basis and the Euclidean algorithm}\label{ss}
Following the notations in Section \ref{EA} for $r_0(x)=f(x)$ and $r_1(x)=g(x)$ we can proceed as in \cite{ABKW90}. In this paper, the authors deal with the rational interpolation problem without multiplicities, to obtain a minimal basis from the Euclidean algorithm. We show here that the same approach works for the general case.

The vector relations
\begin{equation}\label{rec} \begin{pmatrix} r_{i+2}(x)\\ s_{i+2}(x)\\t_{i+2}(x) \end{pmatrix}= \begin{pmatrix} r_i(x)\\s_{i}(x)\\t_{i}(x) \end{pmatrix}- q_{i+1}(x) \begin{pmatrix} r_{i+1}(x)\\s_{i+1}(x)\\t_{i+i}(x) \end{pmatrix}
\end{equation}
with the initial conditions\begin{equation}\label{init}\begin{pmatrix} r_0(x)\\ s_{0}(x)\\t_{0}(x) \end{pmatrix}= \begin{pmatrix} f(x)\\0\\1 \end{pmatrix}, \quad \begin{pmatrix} r_1(x)\\s_{1}(x)\\t_{1}(x) \end{pmatrix}= \begin{pmatrix} g(x)\\ 1\\0 \end{pmatrix}
 \end{equation}
produce a sequence of elements in the kernel of $\begin{pmatrix} 1&-g(x) &-f(x) \end{pmatrix}.$

\begin{proposition}\label{basis}
For $0\leq i\leq N-1,$ the set $$\{(r_i(x),s_i(x),t_i(x)),\, (r_{i+1}(x),s_{i+1}(x),t_{i+1}(x))\}\subset  \K[x]^3$$ is a basis of the kernel of $\varphi$. Also, $\{(r_{i}(x), s_{i}(x)),\,(r_{i+1}(x), s_{i+1}(x))\}$ is a basis of $Y.$
\end{proposition}
\begin{proof}
By induction on $i,$ the case $i=0$ being given by \eqref{init}. For the general case, we just have to show that  one can replace the triplet
$(r_i(x),s_i(x),t_i(x))$ in the basis $\{(r_i(x),s_i(x),t_i(x)),\, (r_{i+1}(x),s_{i+1}(x),t_{i+1}(x))\}$ with $(r_{i+2}(x),s_{i+2}(x),t_{i+2}(x))$ and still generate the same kernel. But this follows straightforwardly thanks to \eqref{rec}, which proves the first part of the claim. The rest holds by projecting onto the first two coordinates the previous result,  and using Remark \ref{unique}.
\end{proof}
From \eqref{rec} we deduce that, for  $i=1,\dots, N,$
 \begin{multline}\label{pristra}\begin{pmatrix} \deg r_{i}(x) & \deg r_{i+1}(x)\\ \deg s_{i}(x) &\deg s_{i+1}(x) \end{pmatrix}=\\
\begin{pmatrix} n-(\deg q_1(x)+\cdots +\deg q_{i}(x)) &  n-(\deg q_1(x)+\cdots +\deg q_{i+1}(x))\\ \deg q_1(x)+\cdots +\deg q_{i-1}(x) &  \deg q_1(x)+\cdots +\deg q_{i}(x) \end{pmatrix} .\end{multline}

\begin{theorem}\label{mtt}
Let $i\in\{1,\ldots, N\}$ be such that
\begin{equation}\label{ristra} n-\sum_{j=1}^{i}\deg q_j(x)\geq \sum_{j=1}^{i-1}\deg q_j(x) \mbox{ and } \sum_{j=1}^{i}\deg q_j(x) \geq
 n- \sum_{j=1}^{i+1} \deg q_{j}(x)\end{equation}
or equivalently
\[ 2\left( \sum_{j=1}^{i-1} \deg q_j(x)\right) +\deg q_{i}(x) \leq n \leq  2\left( \sum_{j=1}^{i} \deg q_j(x)\right) +\deg q_{i+1}(x), \]
then   $\{(r_{i}(x),  s_{i}(x)),\,(r_{i+1}(x), s_{i+1}(x))\}$ is a minimal basis of $Y.$
\end{theorem}
\begin{proof}
From Proposition \ref{basis} we know that $\{(r_{i}(x),  s_{i}(x)),\,(r_{i+1}(x), s_{i+1}(x))\}$ is a  basis of $Y.$ For $j=0,1,$ let
$\delta_j=\max \{\deg r_{i+1-j}(x), \deg s_{i+1-j}(x) \}.$

From \eqref{pristra} and \eqref{ristra}, we deduce that
$$\delta_1+\delta_2=\left(n-\sum_{j=1}^{i}\deg q_j(x)\right)+\left(\sum_{j=1}^{i}\deg q_j(x)\right)=n.
$$
The claim now follows thanks to Remark \ref{any}.

\end{proof}

\begin{example}
\label{example1}
Consider
\[
\begin{array}{lll}
x_1=0, &y_{1,0}=-2 & \\
x_2=2, &y_{2,0}=6& \\
x_3=-1, &y_{3,0}=-3 \\
x_3=-1, & y_{3,1}=3 ,
\end{array}
\]
\end{example}

In this case, we have $N=3$ and the sequences of polynomials produced by the Euclidean Algorithm are
\begin{table}[h!]
  \begin{center}
    %\caption{}
    %\label{}
    \begin{tabular}{c|c|c|c|c}
      $i$ & $0$ & $1$ & $2$ & $3$ \\
			\hline
      $r_i(x)$ &$x^4-3x^2-2x$ & $x^3-2$ & $-3x^2$ &$-2$   \\
      \hline
      $s_i(x) $& $0 $& $1$ & $-x$ &$\frac{-x^2}{3}+1$  \\
			%\hline
      %$t_i$& $1$ & $0$ & $1$ & $\frac{x}{3}$ & $\frac{x^3-2}{-2}$\\
          \end{tabular}
  \end{center}
\end{table}

with $q_1(x)=x$, $q_2(x)=-\frac{x}{3} $, $q_3(x)=\frac{3x^2}{2}$.

So, $\mu_1=\mu_2=2,\, \{(r_2(x),s_2(x)),\,(r_3(x),s_3(x))\}$ is a minimal basis, the $\delta$- minimal degree is $2$ and the fractions $\frac{-2-3\lambda x^2}{\frac{-x^2}{3}+1-\lambda x}$ define the  minimal interpolating functions for all $\lambda\in \K$, $\lambda\neq -\frac{1}{6}, -\frac{2}{3}$.

The interpolating functions $y(x)=\frac{a(x)}{b(x)}$ with $\delta (y(x))=\delta >2$ can be parameterized as
\[
y(x)=\frac{(\lambda_0+\cdots +\lambda_{\delta-2}x^{\delta-2})r_2(x)+ (\lambda'_0+\cdots +\lambda'_{\delta-3}x^{\delta-3} + x^{\delta-2} )r_3(x)  }{(\lambda_0+\cdots +\lambda_{\delta-2}x^{\delta-2})s_2(x)+ (\lambda'_0+\cdots +\lambda'_{\delta-3}x^{\delta-3} + x^{\delta-2})s_3(x)}
\]
provided that $(\lambda_0+\cdots +\lambda_{\delta-2}x_i^{\delta-2})s_2(x_i)+ (\lambda'_0+\cdots +\lambda'_{\delta-3}x_i^{\delta-3} + x_i^{\delta-2})s_3(x_i) \neq 0$, for $i=1, 2, 3$.

The index $i$ in Theorem \ref{mtt} is not unique, although there are at most two possible choices for it as the following cautionary example shows.
\begin{example}
\label{example2}
Take the interpolating data $(1,1)$, $(-1,1)$, $(2,-14)$, $(-2,-14)$, $(3,1)$, $(-3,1)$. In this case, $N=3,$ and  the Euclidean Algorithm gives
\begin{table}[h!]
  \begin{center}
    %\caption{}
    %\label{}
    \begin{tabular}{c|c|c|c|c}
      $i$ & $0$ & $1$ & $2$ & $3$ \\
			\hline
      $r_i(x)$ &$x^6-14x^4+49x^2-36 $ & $ x^4-10x^2+10 $ & $-x^2+4 $ &$-4$   \\
      \hline
      $s_i (x)$& $0 $& $1$ & $-x^2+4 $ &$ -x^4+2x^2-3$  \\
			%\hline
      %$t_i$& $1$ & $0$ & $1$ & $ $ & $ $\\
          \end{tabular}
  \end{center}
\end{table}
\noindent with $q_1(x)=x^2-4$, $q_2(x)=-x^2+6$, $q_3(x)=\frac{x^2-4}{14}$. So, $\mu_1=2$, $\mu_2=4$ and
 $\{(r_1(x),s_1(x)),\, (r_2(x),s_2(x))\}$ is a minimal basis. The $\delta$-minimal degree is $4,$ and
  \[\frac{x^4-10x^2+10 +(\lambda_0+\lambda_1x+\lambda_2x^2)(-x^2+4)}{1+(\lambda_0+\lambda_1x+\lambda_2x^2)(-x^2+4)} \]
   parameterize the minimal rational interpolating functions, provided that the denominators do not vanish in the $x_i's$.

Note that another minimal basis is be given by $\{(r_2(x),s_2(x)),\, (r_3(x),s_3(x))\}.$

\end{example}

\begin{example}
\rm{
In the generic case; that is, if all the quotients in the Euclidean Algorithm have degree one, then $N=n$, $\deg r_i(x)=n-i,$ and $\deg s_i(x)=i-1$ for $i=1,\dots N$. From Theorem \ref{mtt} we deduce straightforwardly that
\begin{enumerate}
\item[(a)] If $n=2k$ then $\{(r_k(x),s_k(x)),\, (r_{k+1}(x),s_{k+1}(x))\}$ is a minimal basis and the $\delta$-minimal degree is $k$.
\item[(b)] If $n=2k+1$ then {\bf both} $\{(r_k(x),s_k(x)), (r_{k+1}(x),s_{k+1}(x))\}$ and \newline $\{(r_{k+1}(x),s_{k+1}(x)), (r_{k+2}(x),s_{k+2}(x))\}$ are minimal bases, and the $\delta$-minimal degree is either $k$ or $k+1$.
\end{enumerate}
}
\end{example}

\begin{example}
\rm{
Consider the data $(-1,-3)$, $(0,-2)$, $(1,-1)$ and $(2,6)$. This is a generic case with an even number of pairs for $r_0(x)=f(x)=x^4-2x^3-x^2+2x$ and $r_1(x)=g(x)=x^3-2$.

\begin{table}[h!]
  \begin{center}
    %\caption{}
    %\label{}
    \begin{tabular}{c|c|c|c|c|c}
      $i$ & $0$ & $1$ & $2$ & $3$ & $4$ \\
			\hline
      $r_i(x)$ &$x^4-2x^3-x^2+2x$ & $x^3-2$ & $-x^2+4x-4 $ & $12x-18 $ & $\frac{-1}{4} $  \\
      \hline
      $s_i(x) $& $0 $& $1$ & $-x+2$ &$ -x^2-2x+9$ &$\frac{-8x^2+4x+3}{24}$  \\
			%\hline
      %$t_i$& $1$ & $0$ & $1$ & $\frac{x}{3}$ & $\frac{x^3-2}{-2}$\\
          \end{tabular}
  \end{center}
\end{table}

The family of functions $\frac{12x-18+\lambda (-x^2+4x-4)}{-x^2-2x+9 +\lambda (-x+2)}$ with $\lambda\neq \frac{-10}{3},\frac{-9}{2},-6$  interpolate with $\delta$-minimal degree equal to $2$.
}
\end{example}

\bigskip
\section{$\mu$-basis for polynomial parameterizations of plane curves}\label{mb}

The approach used to find a minimal basis for the $\delta$-degree actually allows us to compute a $\mu$-basis for the parametrization of a polynomial planar curve, and also characterize the value of $\mu$ in terms of the ``critical'' degree \eqref{ristra}  arising in the Extended Euclidean Algorithm. Definitions and basic properties and applications of $\mu$-basis, in the general setting of rational curves can be found in \cite{CSCh} (see also \cite{HHK17}). We briefly recall them here for polynomial parametrizations.

Let $r_0(x), r_1(x)\in\K[x],$ and consider the polynomial parametrization of an affine plane curve given by

\begin{equation}\label{param}
\begin{array}{ccc}
\K& \to &\K^2\\
t & \mapsto & (r_0(t),r_1(t)).
\end{array}
\end{equation}

Denote with $I$ the ideal in the polynomial ring $\K [ x,T_0,T_1 ]
$ defined by $T_0-r_0(x)$ and $T_1-r_1(x).$

Let $I_{i,j}$ the set of elements of $I$ with degree at most $i$  in $x$ and total degree at most $j$ in $T_0,T_1$. It is easy to see that a polynomial $a(x)T_0+b(x)T_1 +c(x)$ with $a(x), b(x), c(x)\in \K [x]$ is in $I_{\ast,1}$ if, and only if, $a(x)r_0(x)+b(x)r_1(x)+c(x)=0,$ i.e. $(a(x), b(x), c(x))$ is a syzygy of $(r_0(x), r_1(x), 1).$ This implies that there exists an isomorphism of $\K [x]$-modules
\[\begin{array}{cll}
I_{\ast,1} & \longrightarrow &\Syzy (r_0(x), r_1(x),1) \\
a(x)T_0+b(x) T_1 +c(x)   & \mapsto  & (a(x), b(x), c(x))
\end{array}
\]

Assume that $n=\deg r_0(x) \geq \deg r_1(x).$ Let $\mu$ the smallest integer such that  $I_{\mu,1}\neq 0.$ By homogeneizing and applying Hilbert Syzygy Theorem as in \eqref{HST},  there exist polynomials $p(x,T_0,T_1)\in I_{\mu ,1}$ and $q(x,T_0,T_1)\in I_{n-\mu ,1}$ such that every polynomial $a(x)T_0+b(x)T_1 +c(x)\in I_{*,1}$ can be written uniquely in the form
$$a(x)T_0+b(x)T_1 +c(x)=h_1(x)p(x,T_0,T_1)+h_2(x)q(x,T_0,T_1)$$ with $h_1(x),h_2(x)\in \K [x]$. The polynomials $p(x,T_0,T_1),\,q(x,T_0,T_1)$ are called a {\em $\mu$-basis} of the parametrization \eqref{param}.

Homogenizing \eqref{param}, we have a map
\begin{equation}\label{paramh}
\begin{array}{ccc}
\P^1& \to &\P^2\\
(t_0:t_1) & \mapsto & (r_0(t_0:t_1):r_1(t_0:t_1):t_1^n).
\end{array}
\end{equation}
whose image is a projective plane curve. We denote with
\begin{multline*}
 \mathcal{I}_{\ast ,1}= \{ a(x,z)T_0+b(x,z)T_1+c(x,z)T_2  \in \\
  \K[x,z,T_0,T_1,T_2]; a(x,z)r_0(x,z)+b(x,z)r_1(x,z)+c(x,z)z^n=0\}
  \end{multline*}
the $\K[x,z]$-submodule of {\em moving lines following the parametrization} \eqref{paramh}. It is easy to verify that there is an isomorphism of graded $\K[x,z]$-modules
\[\begin{array}{cll}
 \mathcal{I}_{\ast,1} & \longrightarrow &\Syzy (r_0(x,z), r_1(x,z),z^n ) \\
a(x,z)T_0+b(x,z)T_1+c(x,z)T_2    & \mapsto  & \big(a(x,z),b(x,z),c(x,z)\big) .
\end{array}
\]
Again by Hilbert Syzygy Theorem we deduce that  $\Syzy (r_0(x,z), r_1(x,z),z^n ) $ is a free module of rank $2,$ with a homogeneous basis $$p=(p_1(x,z),p_2(x,z),p_3(x,z)),\quad  q=(q_1(x,z),q_2(x,z),q_3(x,z))$$ of degrees $\mu$ and $n-\mu,$ where we assume $\mu \leq n-\mu$. By dehomogenizing this situation, we deduce that   $$\{p_1(x,1)T_0+p_2(x,1)T_1+p_3(x,1), q_1(x,1)T_0+q_2(x,1)T_1+q_3(x,1)\}$$ is a
$\mu$-basis of \eqref{param}.

As in Section \ref{EA}, let $r_i(x)$, $s_i(x)$, $t_i(x)$, for $i=0,\dots, N+1$ the sequences of polynomials from the Extended Euclidean Algorithm starting with $r_0(x)$ of degree $n$ and $r_1(x)$ of lower degree.  From \S \ref{ss}, we deduce that, $(t_i(x),s_i(x),-r_i(x))$ is a syzygy of $(r_0(x),r_1(x),1)$ for all $i=0,\ldots, N.$ The following  result  states the value of $\mu$ in terms of the degrees appearing in the Extended Euclidean Algorithm, and also extracts a $\mu$-basis of \eqref{param} from the sequence of remainders.
\begin{theorem}\label{mu}
There exists $i_{\mu}\in\{0,1,\ldots, N-1\}$ such that
$$\max \{\deg r_{i_{\mu}}(x), \deg s_{i_{\mu}}(x)\} + \max\{\deg r_{i_{\mu}+1}(x), \deg s_{i_{\mu}+1}(x)\} =n .$$
For this index, we have
$$\mu=\min \{\max \{ \deg r_{i_{\mu}}(x), \deg s_{i_{\mu}}(x)\}, \max\{\deg r_{i_{\mu}+1}(x), \deg s_{i_{\mu}+1}(x) \}\},$$
and  moreover $\{(t_{i_{\mu}}(x),s_{i_{\mu}}(x),-r_{i_{\mu}}(x))$, $(t_{i_{\mu}+1}(x),s_{i_{\mu}+1}(x),-r_{i_{\mu}+1}(x))\}$ is a $\mu$-basis of \eqref{param}.
\end{theorem}
\begin{proof}
Choose as $i_\mu$ the one satisfying \eqref{ristra}, then the first part of the claim follows straightforwardly from Theorem \ref{mtt}.  For the rest, suppose without loss of generality that
$\mu=\max \{\deg r_{i_{\mu}}(x), \deg s_{i_{\mu}}(x)\}.$ From  Lemma \ref{lema} we deduce that
$$\mu=\max \{\deg t_{i_\mu}(x), \deg r_{i_{\mu}}(x), \deg s_{i_{\mu}}(x)\},$$ and also that
$n-\mu=\max \{\deg t_{i_\mu+1}(x), \deg r_{i_{\mu}+1}(x), \deg s_{i_{\mu}+1}(x)\}.$ The fact that \linebreak $\{(t_{i_{\mu}}(x),s_{i_{\mu}}(x),-r_{i_{\mu}}(x))$, $(t_{i_{\mu}+1}(x),s_{i_{\mu}+1}(x),-r_{i_{\mu}+1}(x))\}$ is a $\mu$-basis of \eqref{param} follows then straightforwardly from Proposition \ref{basis}.
\end{proof}

\begin{example}
{\rm For the planar affine curve  parameterized by $(6x^2-4x^4,4x-4x^3)$ we have

\begin{table}[h!]
  \begin{center}
    %\caption{}
    %\label{}
    \begin{tabular}{c|c|c|c|c}
      $i$ & $0$ & $1$ & $2$ & $3$\\
			\hline
      $r_i(x)$ &$-4x^4+6x^2$ & $-4x^3+x$ & $2x^2$ & $4x$ \\
      \hline
      $s_i (x)$&  $0$ & $1$ & $-x$ & $1-2x^2$ \\
			\hline
      $t_i(x)$& $1$ & $0$ & $1$ & $2x$\\
          \end{tabular},
  \end{center}
\end{table}
We deduce then that
$\mu=2$ and that $\{T_0-xT_1-2x^2,\, 2xT_0+(1-2x^2)T_1-4x\}$ is a $\mu$-basis of the parametrization.

From the above, we have that for the projective planar curve parametrized by $(6x^2z^2-4x^4,4xz^2-4x^3,z^4)$, the free $\K[x,z]$-module of moving lines following the parametrization is generated  by $z^2T_0-xzT_1-2x^2T_2$ and $2xzT_0+(z^2-2x^2)T_1-4xzT_2$.
}
\end{example}

\begin{example}
{\rm For the planar affine curve parameterized by $(x^n,x^m)$ with $m\leq n,$ we have

\begin{table}[h!]
  \begin{center}
    %\caption{}
    %\label{}
    \begin{tabular}{c|c|c|c}
      $i$ & $0$ & $1$ & $2$  \\
			\hline
      $r_i(x)$ &$x^n$ & $x^m$ & $0$ \\
      \hline
      $s_i(x) $&  $0$ & $1$ & $-x^{n-m}$ \\
			\hline
      $t_i(x)$& $1$ & $0$ & $1$
 \end{tabular}
  \end{center}
\end{table}
From this table we deduce that $\mu=\min(m,n-m),$ and $\{T_1-x^m,\,T_0-x^{n-m}T_1\}$ is a $\mu$-basis of the parametrization.

For the corresponding projective plane curve parametrized by $(x^n,x^mz^{n-m},z^n)$, the module of moving lines following the parametrization is generated by $z^mT_1-x^mT_2$ and $z^{n-m}T_0-x^{n-m}T_1$.
}
\end{example}

\section{The $\kappa$-degree}\label{ccappa}
Here we use again  the notations of Section \ref{EA}, and set $r_0(x)=f(x)$ and $r_1(x)=g(x)$. The following result is obtained straightforwardly from the characterization given in \eqref{eq} of the interpolating functions, and Lemma \ref{emes}.

\begin{theorem}
\label{kappa}
For every rational function $y(x)$ satisfying \eqref{interpolant}, there exists a unique set of polynomials $m_0(x),\dots, m_{N+1}(x)$ such that
\begin{equation}\label{yy}
y(x)=\frac{\sum_{i=0}^{N+1} m_i(x)r_i(x)}{\sum_{i=0}^{N+1} m_i(x)s_i(x)},\mbox{ with } \deg m_i(x) <\deg q_i(x) \mbox{ for } i=1,\dots,N \end{equation}
and $\sum_{i=0}^{N+1} m_i(x_j)s_i(x_j)\neq 0$ for $j=1,\dots l$.
\end{theorem}

Now we are ready to present the main result of our section, which recovers   \cite[Corollary 3.5]{ant88}.
\begin{theorem}\label{kappas}
The set of admissible $\kappa$ 's is
\begin{equation}\label{formula}
\{n-\deg q_k(x),\,s_k(x_i)\neq0 \,\forall i=1,\ldots, l\}\cup\{j\geq n\}.
\end{equation}
The minimal of the admissible degrees is
\begin{equation}\label{inim}
\min_k\{n-\deg q_k(x),\,s_k(x_i)\neq0 \,\forall i=1,\ldots, l\}.
\end{equation}
\end{theorem}
\begin{proof}
%We will work first with fractions of the form
%\begin{equation}\label{y}
%y(x)=\frac{\sum_{i=1}^{N} m_i(x)r_i(x)}{\sum_{i=1}^{N} m_i(x)s_i(x)},
%\end{equation} with  $\deg m_i(x) <\deg q_i(x)$ for $ i=1,\dots,N.$
%From Remark \ref{degrees} we deduce that, for $1\leq k\leq k'\leq N,$ if $m_k(x)\neq0\neq m_{k'}(x),$
%$$
%\begin{array}{ccccccc}
%\deg(r_k(x))&\leq&\deg\big(\sum_{i=k}^{k'}m_i(x)r_i(x)\big)&=&\deg(m_k(x)r_k(x))&<&\deg(r_{k-1}(x)) \\
%\deg(s_{k'}(x))&\leq &\deg\big(\sum_{i=k}^{k'}m_i(x)s_i(x)\big)&=&\deg(m_{k'}(x)s_{k'}(x)&<&\deg(s_{k'+1}(x)),
%\end{array}
%$$
%with equality in the first (resp. second) line if we choose $m_k(x)$ (resp. $m_{k'}(x)$) as a nonzero element in $\K$.
%So, for $y(x)$ as in \eqref{y}, we have

Given an interpolating rational function we write it, by using Theorem \ref{kappa}, as  $y(x)=\frac{\sum_{i=0}^{N+1} m_i(x)r_i(x)}{\sum_{i=0}^{N+1} m_i(x)s_i(x)}$, with $\deg m_i(x) <\deg q_i(x)$ for  $i=1,\dots,N$.

The sequence of degrees $\{ \deg (m_i(x)r_i(x)),\, m_i(x)\neq 0 \}_{i=0,\dots N+1}$ is decreasing while
 $\{ \deg (m_i(x)s_i(x)),\, m_i(x)\neq 0 \}_{i=0,\dots N+1}$ is increasing. So, if we consider the indexes $k=\min \{i \mbox{ such that } m_i(x)\neq 0\}$ and  $k'=\max \{i \mbox{ such that } m_i(x)\neq 0\}$, then
\begin{equation}\label{cappa} \kappa ( y(x))=(\deg r_k(x)+\deg m_k(x))+(\deg s_{k'}(x) +\deg m_{k'}(x)). \end{equation}
In particular, if we want to minimize this quantity, we may assume both that $m_k(x)$ and $m_{k'}(x)$ are nonzero elements of $\K$.
From \eqref{cappa} and Remark \ref{degrees} we deduce that
$$
\deg r_k(x) + \deg s_{k'}(x) = (n-\deg q_1(x)-\cdots-\deg q_k(x))+(\deg q_{1}(x)+ \cdots +\deg q_{k'-1}(x)).$$
This quantity is equal to
\begin{equation}\label{llist}
\left\{\begin{array}{lcl}
n+(\deg q_{k+1}(x)+\cdots+\deg q_{k'-1}(x)) &\mbox{if} & k'> k+1 \\
n &\mbox{if} & k'=k+1 \\
n-\deg q_k(x) & \mbox{if} &k'=k.
\end{array}
\right.\end{equation}
From here we deduce straightforwardly that the minimum gets reached when $k'=k.$

The set defining \eqref{inim} is never empty as $\frac{r_1(x)}{s_1(x)}=\frac{g(x)}{1},$ the Hermite interpolation polynomial is always an element of it. Let $k_0\leq N$ be the minimum of \eqref{inim}. In particular, $k_0\leq \deg g(x)<n$.
%If $k_0=1,$ then the minimum is achieved by $\frac{r_1(x)}{s_1(x)}=\frac{g(x)}{1},$ the Hermite interpolation polynomial. Otherwise, if $s_{k_0}(x_i)\neq0$ %for all $i=1,\ldots, l,$ then the minimum is achieved by $\frac{r_{k_0}(x)}{s_{k_0}(x)}.$  In any case, the set defining \eqref{inim} is never empty as %$k=1$ is always an element of it.

To prove that any degree $j\geq n$ is feasible,  we proceed as follows: choose any two consecutive indexes $1\leq k< k+1\leq N,$ and let $\lambda\in\K\setminus\{0\}$ be such that
$$\frac{r_k(x)+\lambda r_{k+1}(x)}{s_k(x)+\lambda s_{k+1}(x)}
$$
is irreducible (i.e. the numerator and the denominator do not share any common factor) and the denominator does not vanish in $x_i,\, i=1,\ldots, l.$ This can be done as in the proof of Theorem \ref{nininal}. Then, due to \eqref{llist}, we actually have that  the $\kappa$-degree of this function is $n$. For $j>n$, we choose $m_0(x)$ a general polynomial of degree
$j-\big(n-\deg(s_k(x)+\lambda s_{k+1}(x))\big)$ such that the fraction
$$\frac{m_0(x)r_0(x)+r_k(x)+\lambda r_{k+1}(x)}{m_0(x)s_0(x)+s_k(x)+\lambda s_{k+1}(x)}=\frac{m_0(x)f(x)+r_k(x)+\lambda r_{k+1}(x)}{s_k(x)+\lambda s_{k+1}(x)}
$$
is irreducible. Then, it is straightforward to check that the $\kappa$-degree of this function is equal to $j$, which concludes with the proof of the Theorem.
\end{proof}

As another application of Theorem \ref{kappas}, we obtain a description of the solutions of the rational Hermite interpolation problem   (see \cite[Theorem 2.6]{CDM18} and \cite[Exercise 5.42]{vzgg13}). Recall that this problem consists in, for a given $d\in\N,\, 0\leq d<n,$ decide if there exist, and if so compute, polynomials
$a(x),b(x)$ of degrees bounded by $d$ and $n-d-1$ respectively, such that $\frac{a(x)}{b(x)}$ interpolates the data.
 Observe also that the rational Hermite interpolation problem can be phrased as to decide if there exist, and if so compute, polynomials
$a(x),b(x)$ such that $\frac{a(x)}{b(x)}$ interpolates the data with $\deg a(x) \leq d$ and $\kappa (\frac{a(x)}{b(x)} )\leq n-1$. Observe also that this is equivalent, by Proposition \ref{above}, to decide if there exist such a coprime pair of polynomials $\big(a(x)$, $b(x)\big)$.

\begin{corollary}
For $0\leq d \leq n-1$, let $1\leq k \leq N$ such that
\begin{multline*}
\deg r_{k}(x)  =n-( \deg q_1(x) +\cdots + \deg q_k(x)) \leq d <\\
 \deg r_{k-1}(x)= n-(\deg q_1(x) +\cdots + \deg q_{k-1}(x)).
\end{multline*}

If there is a solution of the  Hermite interpolation problem for the integer $d,$ then it is of the form $\frac{p(x)r_k(x)}{p(x)s_k(x)}$ for some $p(x)\in\K[x]$. The problem is solvable if and only if $r_k(x)$ and $s_k(x)$ are coprime.
\end{corollary}

\bigskip

\begin{example}
{\rm For the data in Example \ref{example1} have $n=4$, $N=3$, $\deg q_1(x)=\deg q_2(x)=1$, $\deg q_3(x)=2$. The minimal $\kappa$-degree is 2 and the minimal solution is $\frac{6}{x^2-3}$. The Hermite polynomial is the other interpolating rational function of degree less than 4 and has degree 3. The Hermite interpolation problem has no solution for $d=2$.}
\end{example}

\begin{example}
{\rm For the data in Example \ref{example2}, $n=6$, the minimal $\kappa$-degree is 4  and the minimal solutions are $x^4-10x^2+10$ and $\frac{4}{x^4-2x^2+3}$. There are no other solutions of degree less than 6. The Hermite interpolation problem has no solution for $d=2$ and $d=3$.}
\end{example}

\begin{example}
{\rm In the generic case with a data of $n$ pairs the minimal $\kappa$-degree is $n-1$ and $g(x)$ is a minimal solution. The fractions, $\frac{r_i}{s_i}$, for $i=1,\dots, n$, have  degree $n-1$ and are minimal solutions if $r_i$ and $s_i$ are coprime.}
\end{example}

\bigskip
 \bibliographystyle{alpha}

\begin{thebibliography}{ABKW90}

\bibitem{antand86}
Antoulas, A.C.; Anderson, B. D. Q.
\newblock{\em On the Scalar Rational Interpolation Problem.\/}
\newblock   IMA Journal of Mathematical Control \& Information 3 (1986), 61--88.


\bibitem{ant88}
Antoulas, A.C.
\newblock{\em Rational interpolation and the Euclidean algorithm.\/}
\newblock   Linear Algebra Appl. 188 (1988), 157--171.


\bibitem{ABKW90}
Antoulas, A.C.; Ball J.A.; Kang J.; Willens J.C.
\newblock{\em On the Solution of the Minimal Rational Interpolation Problem.\/}
\newblock   Linear Algebra and Appl. 137/138 (1990), 511--573.

\bibitem{BL00}
Beckermann, B.; Labahn, G.
\newblock{\em Fraction-free computation of matrix rational interpolants and matrix GCDs.\/}
\newblock  SIAM J. Matrix Anal. Appl. 22 (2000), no. 1, 114--144.

\bibitem{ca41}
Cauchy, A. L.
\newblock{\em M\'emoire sur les fonctions altern\'ees et les sommes altern\'ees.\/}
\newblock Exercices d'analyse et de phys. math. (1841), 151--159.

\bibitem{CDM18}
Cortadellas Ben\'{i}tez, T.: D'Andrea, C.; Montoro, E.
\newblock{\em The set of unattainable
points for the rational Hermite interpolation problem. \/}
\newblock Linear Algebra Appl. 538 (2018), 116–-142.


\bibitem{CSCh}
Cox, D.A.; Sederberg, T.W.; Chen, F.
\newblock{\em The moving line ideal basis of planar rational curves.\/}
\newblock  Comput. Aided Geom. Design 15 (1998), no. 8, 803--827.

\bibitem{DKS15}
D'Andrea, C.; Krick, T.; Szanto, A.
\newblock{\em Subresultants, Sylvester sums and the rational interpolation problem.\/}
\newblock J. Symbolic Comput. 68 (2015), 72--83.

\bibitem{HHK17}
Hong, H.; Hough, Z.;  Kogan I. A.
\newblock{\em Algorithm for computing $\mu$-bases of univariate polynomials.\/}
\newblock J. Symbolic Comput. 80 (2017), part3, 844--874.

%\bibitem[HHKZ18]{HHKZ18}
%Hong, H.; Hough, Z.;Kogan I. A.; Li, Z.
%\newblock{\em Degree-optimal moving frames for rational curves.\/}
%\newblock  {\tt arXiv:1703.03014}
\bibitem{Kah69}
Kahng, S. W.
\newblock{\em Osculatory interpolation.\/}
\newblock Math. Comp. 23 (1969), 621--629.

\bibitem{las03}
Lascoux, A.
\newblock{\em Symmetric functions and combinatorial operators on polynomials.\/}
\newblock CBMS Regional Conference Series in Mathematics, 99. Published for the Conference Board of the Mathematical Sciences, Washington, DC; American Mathematical Society, Providence, RI, 2003.

\bibitem{ra97}
Ravi, M. S.
\newblock{\em Geometric methods in rational interpolation theory.\/}
\newblock   Linear Algebra Appl. 258 (1997), 159--168.

\bibitem{pre53}
Predonzan, A.
\newblock{\em Su una formula d'interpolazione per le funzioni razionali.\/}
\newblock Rend. Sem. Mat. Univ. Padova 22(1953). 417--425.

\bibitem{ros45}
Rosenhain, G.
\newblock{\em Neue Darstellung der Resultante der Elimination von z aus zwei algebraische Gleichungen.\/}
Crelle J. 30 (1845) 157--165.


\bibitem{Sal62}
Salzer, H. E.
\newblock{\em Note on osculatory rational interpolation.\/}
\newblock Math. Comp. 16 1962 486--491.

\bibitem{TF00}
Tan, J.; Fang, Y.
\newblock{\em Newton-Thiele's rational interpolants.\/}
\newblock Computational methods from rational approximation theory (Wilrijk, 1999). Numer. Algorithms 24 (2000), no. 1--2, 141--157.

\bibitem{vzgg13}
Von zur Gathen, J.; Gerhard, J.
\newblock{\em  Modern computer algebra. \/}
\newblock Third edition. Cambridge
University Press, Cambridge, 2013.

\bibitem{Wuy75}
Wuytack, L.
\newblock{\em On the osculatory rational interpolation problem. \/}
Math. Comput. 29 (1975), 837--843.

\end{thebibliography}

\end{document}